\documentclass{amsart}
\usepackage{amsmath,amsthm,amssymb,IMjournal}

\newcommand{\Hn}{ {\mathbb{H}^n} }
\newcommand{\Bn}{ {\mathbb{B}^n} }
\newcommand{\Rn}{ {\mathbb{R}^n} }

\theoremstyle{plain}

\newtheorem{lemma}[equation]{Lemma}

\allowdisplaybreaks[4]

\begin{document}
\newtheorem{The}{Theorem}[section]

\numberwithin{equation}{section}

\title{Lipschitz constants for a hyperbolic type metric under M\"obius transformations}

\author{\|Yinping |Wu|, Hangzhou,
        \|Gendi |Wang*|, Hangzhou,
        \|Gaili |Jia|, Hangzhou,
        \|Xiaohui |Zhang|, Hangzhou
        }

\rec {August XX, 2023}


\abstract
  Let $D$ be a nonempty open set in a metric space $(X,d)$ with $\partial D\neq \emptyset$. Define
\begin{equation*}
h_{D,c}(x,y)=\log\left(1+c\frac{d(x,y)}{\sqrt{d_D(x)d_D(y)}}\right),
\end{equation*}
where $d_D(x)=d(x,\partial D)$ is the distance from $x$ to the boundary of $D$.
For every $c\geq 2$, $h_{D,c}$ is a metric.
In this paper, we study the sharp Lipschitz constants for the metric $h_{D,c}$ under M\"obius transformations of the unit ball, the upper half space, and  the punctured unit ball.
\endabstract

\keywords
  Lipschitz constant, hyperbolic type metric, M\"obius transformation
\endkeywords

\subjclass
51M10, 30C65
\endsubjclass


\section{Introduction}\label{intro}

In geometric function theory, numerous hyperbolic type metrics including the quasihyperbolic metric, distance ratio metric, and Apollonian metric, as extensions of the classical hyperbolic metric have been used and become standard tools \cite{bm,gh0,gh,go,gp,hkvb,V2}. Very recently, Mocanu \cite{mocanu} studied metric-preserving functions with respect to intrinsic metrics of hyperbolic type which yield new metrics. Conformal invariants and conformally invariant metrics are the key notions of geometric function theory and of quasiconformal mapping theory \cite{ah,avv,bm,V2}. However, for hyperbolic type metrics we cannot anymore expect the same invariance properties as the classical hyperbolic metric under conformal mappings. Fortunately, the hyperbolic type metrics still have some type of "near-invariance" or "quasi-invariance"  as desirable feature under (quasi)conformal mappings.
It is natural to ask what the Lipschitz constants are for these hyperbolic type metrics under (quasi)conformal mappings. For example, Gehring, Osgood,
and Palka proved that these metrics are not changed by more than a factor 2 under M\"obius transformations \cite{go,gp}, and the similar results for the other hyperbolic type metrics can be found in \cite{CHKV,hkvb,HVW,HVZ,imsz,jwz,klvw,KMS,MS2,WXV,XWZ}.

Let $D$ be a nonempty open set in a metric space $(X,d)$ with $\partial D\neq \emptyset$. Define
\begin{equation*}
h_{D,c}(x,y)=\log\left(1+c\frac{d(x,y)}{\sqrt{d_D(x)d_D(y)}}\right),
\end{equation*}
where $d_D(x)=d(x,\partial D)$ is the distance from $x$ to the boundary of $D$.
Dovgoshey, Hariri, and Vuorinen \cite{dhv} proved that $h_{D,c}$ is a metric for every $c\geq 2$ and the constant $c=2$ is the best possible in the sense that for each $c\in(0,2)$, there exists an open set $D$ such that the triangle inequality for $h_{D,c}$ fails. Before the work \cite{dhv}, H\"ast\"o \cite{hasto} has studied $h_{D,c}$ in the case when $X=\Rn$ and $D=\Rn\setminus\{0\}$. The metric $h_{D,c}$, as other hyperbolic type metrics, has applications in the study of geometric function theory,  such as estimates of the Kobayashi and quasihyperbolic distances \cite{na}, regularity of quasiconformal maps \cite{dhv}. In particular, the distortion of the metric $h_{D,c}$ under quasiconformal mappings is as follows \cite[Theorem 4.9]{dhv}: for all $c>0$, there exists a constant $C>1$ such that
$$h_{f(D),c}(f(x),f(y))\leq C\max\{h_{D,c}(x,y),h_{D,c}(x,y)^{\alpha}\},$$
where $D\subsetneqq\Rn$ is a uniform domain, $f:D\to f(D)$ is a $K$-quasiconformal mapping, and $\alpha=K^{1/(1-n)}$. Quasiconformal mappings are widely studied by use of various hyperbolic type metrics these days. The interested readers are referred to, e.g., \cite{CHKV,gh0,go,gp,hkvb,HVW,HVZ}.

In this paper we investigate the Lipschitz constants for the metric $h_{D,c}$ under M\"obius transformations of the unit ball, the upper half-space, and the punctured unit ball, respectively. The main results are as follows.

\begin{The}\label{thm:b2b}
Let $a\in\mathbb{B}^n$ and $f:\mathbb{B}^n\rightarrow\mathbb{B}^n=f\left(\mathbb{B}^n\right)$  be a M\"{o}bius transformation with $f\left(a\right)=0$. Then, for all  $x,y\in \mathbb{B}^n$,
\begin{equation*}
\frac{1}{1+|a|}\,h_{\mathbb{B}^n, c}\left( x, y \right)\leq h_{\mathbb{B}^n, c}\left( f\left(x\right), f\left(y\right) \right)\leq \left(1+|a|\right)\,h_{\mathbb{B}^n, c}\left( x, y \right),
\end{equation*}
where the constant $1+|a|$  is the best possible.
\end{The}

\begin{The}\label{thm:b2h}
Let $f:\mathbb{B}^n\rightarrow\mathbb{H}^n=f(\mathbb{B}^n)$ be a M\"{o}bius transformation. Then, for all $x,y\in \mathbb{B}^n$\,,
\begin{align*}
h_{\mathbb{B}^n, c}\left( x, y \right)\leq h_{\mathbb{H}^n, c}\left( f\left(x\right), f\left(y\right) \right)\leq 2\,h_{\mathbb{B}^n, c}\left( x, y \right)\,,
\end{align*}
where the constants $1$ and $2$ are the  best possible.
\end{The}

\begin{The}\label{thm:h2h}
Let $f:\mathbb{H}^n\rightarrow\mathbb{H}^n=f(\mathbb{H}^n)$ be a M\"{o}bius transformation.
Then, for all $x,y\in \mathbb{H}^n$,
\begin{align*}
h_{\mathbb{H}^n, c}\left( f\left(x\right), f\left(y\right) \right)= h_{\mathbb{H}^n, c}\left( x, y \right).
\end{align*}
\end{The}

\begin{The}\label{thm:pb}
Let $a\in\mathbb{B}^n$ and $f:\mathbb{B}^n\rightarrow\mathbb{B}^n=f\left(\mathbb{B}^n\right)$ be a  M\"obius transformation  with $f\left(0\right)=a$.
Then, for all $x,y\in\mathbb{B}^n\backslash\{0\}$, we have
\begin{align*}
h_{\mathbb{B}^n\backslash\{a\},c}\left(f\left(x\right),f\left(y\right)\right)\leq \left(1+|a|\right)h_{\mathbb{B}^n\backslash\{0\},c}\left(x,y\right).
\end{align*}
\end{The}

\section{M\"obius transformations and some inequalities}

We use the standard notations as in \cite{avv}.  The standard basis in the Euclidean space $\Rn$, $n>2$, is denoted by $\{e_1,e_2,\dots, e_n\}$. A point $x\in\Rn$ is represented as $(x_1,x_2,\dots, x_n)$.  The ball centered at $x\in\Rn$ with radius $r>0$ is $B^n(x,r)=\{y\in\Rn:|x-y|<r\}$ and the sphere with the same center and radius is $S^{n-1}(x,r)=\{y\in\Rn:|x-y|=r\}$. We use the abbreviation $\Bn=B^n(0,1)$. The upper half space is denoted by $\Hn=\{x\in\Rn:x_n>0\}$ and $\overline\Hn=\Hn\cup\{\infty\}$. For basic facts about M\"obius transformations the reader is referred to \cite{b,r,V2}.
We denote $x^*=x/|x|^2$ for $x\in\Rn\setminus\{0\}$, and $0^*=\infty$, $\infty^*=0$.
For $a\in \mathbb{B}^n\setminus\{0\}$, let
$$\sigma_a\left(x\right)=a^*+r^2\left(x-a^*\right)^*,\,\, r^2=|a|^{-2}-1$$
be the inversion in the sphere $S^{n-1}\left(a^*,r\right)$. Then $\sigma_a\left(a\right)=0$, $\sigma_a\left(a^*\right)=\infty$, and
(\cite[(3.1.5)]{b})
\begin{align}\label{eq2}
|\sigma_a\left(x\right)-\sigma_a \left(y\right)|=\frac{r^2|x-y|}{|x-a^*||y-a^*|}\,.
\end{align}

The following lemmas show some fundamental results on M\"obius transformations.

\begin{lemma}\label{le9}
\cite[Theorem 3.5.1]{b}
Let $a\in \mathbb{B}^n\setminus\{0\}$, and let $f$ be a M\"obius transformation of the unit ball with $f\left(a\right)=0$. Then
\begin{align*}
f\left(x\right)=\left(A\circ\sigma_a\right)\left(x\right),
\end{align*}
where $A$ is an orthogonal transformation.
\end{lemma}

\begin{lemma}\label{yl18}
\cite[Theorem 4.4.8]{r}
Let $f:\overline{\mathbb{R}^n}\to \overline{\mathbb{R}^n}$ be a M\"{o}bius transformation with $f(\Bn)=\Bn$. Then $f$ is an orthogonal transformation from $\mathbb{R}^n$ onto itself if and only if $f\left(0\right)=0$.
\end{lemma}

\begin{lemma}\label{yl19}
\cite[Theorem 4.3.2]{r}
Let $f:\overline{\mathbb{R}^n}\rightarrow \overline{\mathbb{R}^n}$ be a M\"{o}bius transformation.
Then $f$ is a similarity from $\mathbb{R}^n$ onto itself if and only if $f\left(\infty\right)=\infty$.
\end{lemma}

The following inequalities are useful in the proofs of main results.

\begin{lemma}\label{yl17}\cite[Lemma 3.2]{SVW}
Let $a\,, b \in \mathbb{B}^n$. Then
\begin{enumerate}
\item  $|a|^2|b-a^{*}|^2-|b-a|^2=\left(1-|a|^2\right) \left(1-|b|^2\right)\,;$
\item $$\frac{||b|-|a||}{1-|a||b|}\leq \frac{|b-a|}{|a||b-a^*|}\leq\frac{|b|+|a|}{1+|a||b|}.$$
\end{enumerate}
\end{lemma}

\begin{lemma}\label{le 11}
For $a,z\in\mathbb{B}^n$, there holds
\begin{align*}
&|a||z-a^*|\geq1-|a||z|.
\end{align*}
\end{lemma}

\begin{proof}
By the triangle inequality, we have
\begin{align*}
|a||z-a^*|\geq|a|||z|-|a^*||=1-|a||z|.
\end{align*}
\end{proof}


\begin{lemma}\label{le 12}
For $a,z\in\mathbb{B}^n$ with $\frac{1}{2}\leq|z|\leq1$, there holds
\begin{align*}
|z|\left(1-|a||z|\right)\geq\left(1-\frac{|a|}{2}\right)\left(1-|z|\right).
\end{align*}
\end{lemma}

\begin{proof}
It is easy to see that
\begin{align*}
&|z|\left(1-|a||z|\right)-\left(1-\frac{|a|}{2}\right)\left(1-|z|\right)\\
=&2|z|-1+\frac{|a|}{2}-|a||z|^2-\frac{|a||z|}{2}\\
=&2|z|-1+\frac{|a|}{2}-|a||z|-|a||z|^2+\frac{|a||z|}{2}\\
=&2|z|-1-\left(2|z|-1\right)\frac{|a|}{2}-\left(2|z|-1\right)\frac{|a||z|}{2}\\
=&\left(2|z|-1\right)\left(1-\frac{|a|}{2}-\frac{|a||z|}{2}\right)\geq0.
\end{align*}
\end{proof}

For the reference, we record Bernoulli's inequality here:
\begin{align}\label{bql}
r\log\left(1+t\right)\geq \log\left(1+rt\right), \quad r\geq1,\quad t>0.
\end{align}

\section{Proofs of the main results}

\subsection{Proof of Theorem \ref{thm:b2b}}
For $a=0$, Lemma \ref{yl18} implies that $f$  is an orthogonal transformation. Since $h_{\mathbb{B}^n, c}$ is unchanged under orthogonal transformations, for all $x,y\in \mathbb{B}^n$\,, we have
\begin{align*}
h_{\mathbb{B}^n, c}\left( f\left(x\right), f \left(y\right) \right)= h_{\mathbb{B}^n, c}\left( x, y \right)
\end{align*}
as desired.

Now we assume $a\neq0$. Since $f=A\circ\sigma_a$ by Lemma \ref{le9}, we have for all $x, y\in \mathbb{B}^n$
\begin{align*}
 h_{\mathbb{B}^n, c}\left( f\left(x\right), f \left(y\right) \right)= h_{\mathbb{B}^n, c}\left( \sigma_a\left(x\right), \sigma_a \left(y\right) \right)\,.
\end{align*}
The equation \eqref{eq2} implies that
\begin{align*}
|\sigma_a\left(x\right)-\sigma_a \left(y\right)|=\frac{\left(|a|^{-2}-1\right)|x-y|}{|x-a^*||y-a^*|}
\quad\hbox {and}\quad
1-|\sigma_a \left(x\right)|=\frac{|a||x-a^*|-|x-a|}{|a||x-a^*|}\,.
\end{align*}

We first prove the right-hand side of the inequalities. By use of Lemma \ref{yl17} and  Bernoulli's inequality \eqref{bql}, we have
 \begin{align*}
 &h_{\mathbb{B}^n, c}\left( \sigma_a\left(x\right), \sigma_a\left(y\right)\right)\\
 &=\log\left(1+c \frac{|\sigma_a\left(x\right)-\sigma_a \left(y\right)|}{\sqrt{\left(1-|\sigma_a \left(x\right)|\right)\left(1-|\sigma_a \left(y\right)|\right)}}\right)\,\\
 &=\log\left(1+c \frac{|x-y|}{|a|\sqrt{|x-a^*||y-a^*|}} \sqrt{\frac{1}{\left(|a||x-a^*|-|x-a|\right)\left(|a||y-a^*|-|y-a|\right)}}\right)\\
 &=\log\left(1+c \frac{|x-y|}{|a|\sqrt{|x-a^*||y-a^*|}} \sqrt{\frac{\left(|a||x-a^*|+|x-a|\right)\left(|a||y-a^*|+|y-a|\right)}{\left(1-|x|^2\right)\left(1-|y|^2\right)}}\right)\\
 &=\log\left(1+c\frac{|x-y|}{\sqrt{\left(1-|x|^2\right)\left(1-|y|^2\right)}}\sqrt{\left(1+\frac{|x-a|}{|a||x-a^*|}\right)\left(1+\frac{|y-a|}{|a||y-a^*|}\right)}\right)\\
 &\leq\log\left(1+c\frac{|x-y|}{\sqrt{\left(1-|x|^2\right)\left(1-|y|^2\right)}}\sqrt{\left(1+\frac{|x|+|a|}{1+|x||a|}\right)\left(1+\frac{|y|+|a|}{1+|y||a|}\right)}\right)\\
 &=\log\left(1+c\frac{|x-y|}{\sqrt{\left(1-|x|\right)\left(1-|y|\right)}}\sqrt{\frac{\left(1+|a|\right)^2}{\left(1+|x||a|\right)\left(1+|y||a|\right)}}\right)\\
 &\leq\log\left(1+c\frac{\left(1+|a|\right)|x-y|}{\sqrt{\left(1-|x|\right)\left(1-|y|\right)}}\right)\\
 &\leq\left(1+|a|\right)\,\log\left(1+c\frac{|x-y|}{\sqrt{\left(1-|x|\right)\left(1-|y|\right)}}\right)\\
 &=\left(1+|a|\right)\,h_{\mathbb{B}^n, c}\left( x, y \right)\,.
 \end{align*}
Therefore,
\begin{align*}
h_{\mathbb{B}^n, c}\left( f\left(x\right), f \left(y\right) \right)\leq\left(1+|a|\right)\,h_{\mathbb{B}^n, c}\left( x, y \right)\,.
\end{align*}

Next, we  prove the left-hand side of the inequalities.
Since $f\left(x\right)=\left(A\circ\sigma_a\right)\left(x\right)$ by Lemma \ref{le9}, $f^{-1}\left(x\right)=\left(A\circ\sigma_a\right)^{-1}\left(x\right)=\left(\sigma_a\circ A^T\right)\left(x\right)$.
Hence, for all $u, v \in \mathbb{B}^n$, we have
\begin{align*}
h_{\mathbb{B}^n, c}\left( f^{-1}\left(u\right), f^{-1}\left(v\right)\right)
&=h_{\mathbb{B}^n, c}\left( \sigma_a\left(A^T\left(u\right)\right), \sigma_a \left(A^T\left(v\right)\right) \right)\\
&\leq \left(1+|a|\right)h_{\mathbb{B}^n, c}\left(A^T\left(u\right), A^T\left(v\right) \right)\\
&=\left(1+|a|\right) h_{\mathbb{B}^n, c}\left(u, v\right).
\end{align*}
Now let $u=f\left(x\right)$ and $v=f\left(y\right)$, we have
\begin{align*}
h_{\mathbb{B}^n, c}\left( f\left(x\right), f\left(y\right) \right)\geq\frac{1}{1+|a|}h_{\mathbb{B}^n, c}\left( x, y \right)\,.
\end{align*}

For the sharpness of the constants, we set $x=t\frac{a}{|a|}=-y$ with  $t\in\left(0, |a|\right)$\,. Then
\begin{align*}
|\sigma_a\left(x\right)-\sigma_a \left(y\right)|=\frac{2t\left(1-|a|^2\right)}{1-|a|^2t^2}
\end{align*}
and
\begin{align*}
1-|\sigma_a\left(x\right)|=\frac{\left(1-|a|\right)\left(1+t\right)}{1-t|a|}\,, \quad
1-|\sigma_a\left(y\right)|=\frac{\left(1-|a|\right)\left(1-t\right)}{1+t|a|}\,.
\end{align*}
Hence,
\begin{align}\begin{split}\label{eq9}
\lim_{t\rightarrow 0^{+}}\frac{h_{\mathbb{B}^n, c}\left(f\left(x\right), f \left(y\right) \right)}{h_{\mathbb{B}^n, c}\left( x, y \right)}
&=\lim_{t\rightarrow 0^{+}}\frac{h_{\mathbb{B}^n, c}\left( \sigma_a\left(x\right), \sigma_a \left(y\right) \right)}{h_{\mathbb{B}^n, c}\left( x, y \right)}\\
&=\lim_{t\rightarrow 0^{+}}\frac{\log\left(1+c\frac{2t\left(1+|a|\right)}{\sqrt{\left(1-t^2\right)\left(1-|a|^2t^2\right)}}\right)\,}{\log\left(1+c\frac{2t}{1-t}\right)\,}\\
&=\lim_{t\rightarrow 0^{+}}\frac{\left(1+|a|\right)\left(1-t\right)}{\sqrt{\left(1-|a|^2t^2\right)\left(1-t^2\right)}}
=1+|a|\,,
\end{split}\end{align}
It follows from \eqref{eq9} and $\sigma_a\left(\sigma_a\left(x\right)\right)=x$ that
\begin{align*}
\lim_{t\rightarrow 0^{+}}\frac{h_{\mathbb{B}^n, c}\left(f\left(\sigma_a\left(x\right)\right), f \left(\sigma_a\left(y\right)\right) \right)}{h_{\mathbb{B}^n, c}\left( \sigma_a\left(x\right), \sigma_a\left(y\right) \right)}
&=\lim_{t\rightarrow 0^{+}}\frac{h_{\mathbb{B}^n, c}\left( \sigma_a\left(\sigma_a\left(x\right)\right), \sigma_a \left(\sigma_a\left(y\right)\right) \right)}
{h_{\mathbb{B}^n, c}\left( \sigma_a\left(x\right), \sigma_a\left(y\right) \right)}\\
&=\lim_{t\rightarrow 0^{+}}\frac{h_{\mathbb{B}^n, c}\left( x, y\right)}
{h_{\mathbb{B}^n, c}\left( \sigma_a\left(x\right), \sigma_a\left(y\right) \right)}
=\frac{1}{1+|a|}.
\end{align*}
This completes the proof.\hfill$\Box$

\subsection{Proof of Theorem \ref{thm:b2h}}

The right-hand side of the inequalities was proved in \cite[Lemma 2.11]{dhv}.

Now we prove the left-hand side of the inequalities.
Let $\phi_s:\mathbb{H}^n\rightarrow\mathbb{H}^n$ be a similarity with $\phi_s\left(e_n\right)=a=f\left(0\right)$,
and let $\sigma:\mathbb{B}^n\rightarrow\mathbb{H}^n$ be the inversion in the sphere $S^{n-1}\left(-e_n, \sqrt{2}\right)$ with $\sigma\left(0\right)=e_n$.
Then $f^{-1}\circ \phi_s \circ \sigma:\mathbb{B}^n\rightarrow\mathbb{B}^n$ is a M\"{o}bius transformation with $f^{-1}\left(\phi_s\left(\sigma\left(0\right)\right)\right)=f^{-1}\left(\phi_s\left(e_n\right)\right)=f^{-1}\left(a\right)=0$.
By Lemma \ref{yl18}, we see that $\psi_o\equiv f^{-1}\circ \phi_s \circ \sigma$ is an orthogonal transformation
and $f=\phi_s \circ \sigma \circ {\psi_o}^{-1}$.

Since $\sigma$ is an inversion in $S^{n-1}\left(-e_n, \sqrt{2}\right)$, we have
\begin{align*}
\sigma\left(x\right)=-e_n+\frac{2\left(x+e_n\right)}{|x+e_n|^2}.
\end{align*}
Let $\sigma\left(x\right)=\left(\sigma_1\left(x\right), \sigma_2\left(x\right),\cdots,\sigma_n\left(x\right)\right)$\,, then
\begin{align*}
\sigma_n\left(x\right)=\frac{1-|x|^2}{|x+e_n|^2} \,.
\end{align*}
It follows from \eqref{eq2} that
\begin{align*}
|\sigma\left(x\right)-\sigma\left(y\right)|=\frac{2|x-y|}{|x+e_n||y+e_n|}\,.
\end{align*}

By the invariance of $h_{\mathbb{H}^n, c}$ under similarities, we have for all $x, y \in \mathbb{B}^n$\,,
\begin{align*}
&h_{\mathbb{H}^n, c}\left(f\left(x\right), f\left(y\right) \right)\\
&=h_{\mathbb{H}^n, c}\left( \sigma \left( \psi_o^{-1}\left(x\right)\right),  \sigma \left( \psi_o^{-1}\left(y\right)\right) \right)\\
&=\log\left( 1+ c\frac{|\sigma \left( \psi_o^{-1}\left(x\right)\right)-\sigma\left(\psi_o^{-1}\left(y\right)\right)|}{\sqrt{\sigma_n\left(\psi_o^{-1}\left(x\right)\right)\sigma_n\left(\psi_o^{-1}\left(y\right)\right)}}\right)\\
&=\log\left( 1+ c\frac{2|\psi_o^{-1}\left(x\right)-\psi_o^{-1}\left(y\right)|}{\sqrt{\left(1+|\psi_o^{-1}\left(x\right)|\right)\left(1+|\psi_o^{-1}\left(y\right)|\right)}\sqrt{\left(1-|\psi_o^{-1}\left(x\right)|\right)\left(1-|\psi_o^{-1}\left(y\right)|\right)}}\right)\\
&=\log\left( 1+ c\frac{2|x-y|}{\sqrt{\left(1+|x|\right)\left(1+|y|\right)}\sqrt{\left(1-|x|\right)\left(1-|y|\right)}}\right)\\
&\geq\log\left( 1+ c\frac{|x-y|}{\sqrt{\left(1-|x|\right)\left(1-|y|\right)}}\right)\\
&=h_{\mathbb{B}^n, c}\left( x, y \right)\,.
\end{align*}

For the sharpness of the constants, set $x=-y=te_1$\ with $t\in\left(0,1\right)$. Then we have
\begin{align*}
h_{\mathbb{H}^n, c}\left(\sigma\left(x\right), \sigma\left(y\right) \right)
=\log\left( 1+ c\frac{4t}{1-t^2}\right)
\quad\hbox {and}\quad
h_{\mathbb{B}^n, c}\left( x, y \right)
=\log\left( 1+ c\frac{2t}{1-t}\right)\,.
\end{align*}
Hence,
\begin{align*}
&\lim_{t\rightarrow 0^{+}}\frac{h_{\mathbb{H}^n, c}\left(f\left(x\right), f \left(y\right) \right)}{h_{\mathbb{B}^n, c}\left( x, y \right)}
=\lim_{t\rightarrow 0^{+}}\frac{h_{\mathbb{H}^n, c}\left(\sigma\left(x\right), \sigma\left(y\right) \right)}{h_{\mathbb{B}^n, c}\left( x, y \right)}
=\lim_{t\rightarrow 0^{+}}\frac{2}{1+t}=2
\end{align*}
and
\begin{align*}
\lim_{t\rightarrow 1^{-}}\frac{h_{\mathbb{H}^n, c}\left(f\left(x\right), f \left(y\right) \right)}{h_{\mathbb{B}^n, c}\left( x, y \right)}
&=\lim_{t\rightarrow 1^{-}}\frac{h_{\mathbb{H}^n, c}\left(\sigma\left(x\right), \sigma \left(y\right) \right)}{h_{\mathbb{B}^n, c}\left( x, y \right)}\\
&=\lim_{t\rightarrow 1^{-}}\frac{\log\left(1-t^2\right)}{\log \left(1-t\right)}
=\lim_{t\rightarrow 1^{-}}\frac{2t}{1+t}=1.
\end{align*}
This completes the proof.\hfill$\Box$

\subsection{Proof of Theorem \ref{thm:h2h}}
We assume that $f\left(a\right)=\infty$ for some $a\in\partial \mathbb{H}^n$. If $a=\infty$, Lemma \ref{yl19} implies that $f$ is a similarity.
By the invariance of $h_{\mathbb{H}^n, c}$ under similarities, we have  for all $x,y\in \mathbb{H}^n$,
\begin{align*}
h_{\mathbb{H}^n, c}\left( f\left(x\right), f\left(y\right) \right)= h_{\mathbb{H}^n, c}\left( x, y \right).
\end{align*}

Now we assume $a\neq \infty$.
Let $\varphi_t:\overline{\mathbb{H}^n}\rightarrow\overline{\mathbb{H}^n}$ be a translation with $\varphi_t\left(x\right)=x-a$, then $\varphi_t\left(a\right)=0$.
Let $\sigma:\mathbb{H}^n\rightarrow\mathbb{H}^n$ be the inversion in the sphere $S^{n-1}\left(0, 1\right)$, then $\sigma\left(0\right)=\infty$ and $\sigma\left(\infty\right)=0$.
We see that $f\circ {\varphi_t} ^{-1}\circ \sigma^{-1}:\mathbb{H}^n\rightarrow\mathbb{H}^n$ is a  M\"{o}bius transformation with
$f\left( {\varphi_t} ^{-1}\left( \sigma^{-1}\left(\infty\right)\right)\right)=f\left({\varphi_t}^{-1}\left(0\right)\right)=f\left(a\right)=\infty $.
It follows from Lemma \ref{yl19} that $\phi_s\equiv f\circ {\varphi_t} ^{-1}\circ \sigma^{-1}$ is a similarity, and hence $f=\phi_s\circ\sigma\circ\varphi_t$.

Since $\sigma$ is the inversion in the sphere $S^{n-1}\left(0, 1\right)$, we have
\begin{align*}
\sigma\left(x\right)=\frac{x}{|x|^2} \quad \hbox{and} \quad \sigma_n\left(x\right)=\frac{x_n}{|x|^2}\,.
\end{align*}
By \eqref{eq2}, we have
\begin{align*}
|\sigma\left(x\right)-\sigma\left(y\right)|=\frac{|x-y|}{|x||y|}\,.
\end{align*}

Since $h_{\mathbb{H}^n, c}$ is unchanged under similarities, we have  for all $x,y\in \mathbb{H}^n$,
\begin{align*}
h_{\mathbb{H}^n, c}\left( f\left(x\right), f\left(y\right) \right)
&=h_{\mathbb{H}^n, c}\left( \sigma\left(\varphi_t\left(x\right)\right), \sigma\left(\varphi_t\left(y\right)\right) \right)\\
&=\log\left( 1+ c\frac{|\sigma\left(\varphi_t\left(x\right)\right)-\sigma\left(\varphi_t\left(y\right)\right)|}{\sqrt{\sigma_n\left(\varphi_t\left(x\right)\right)\sigma_n\left(\varphi_t\left(y\right)\right)}}\right)\\
&=\log\left( 1+ c\frac{|\varphi_t\left(x\right)-\varphi_t\left(y\right)|}{\sqrt{{\left(\varphi_t\right)}_n\left(x\right){\left(\varphi_t\right)}_n\left(y\right)}}\right)\\
&=\log\left( 1+ c\frac{|x-y|}{\sqrt{x_ny_n}}\right)\\
&= h_{\mathbb{H}^n, c}\left( x, y \right).
\end{align*}
This completes the proof.\hfill$\Box$


\subsection{Proof of Theorem \ref{thm:pb}}

Since $h_{D,c}$-metric is preserved under orthogonal transformations, it follows from Lemma \ref{le9} that, for $x,y,a\in\mathbb{B}^n$,
it suffices to prove the theorem for $f=\sigma_a$.

Without loss of generality, we assume that $|y|\leq|x|$. Let $D=\mathbb{B}^n\setminus\{0\}$ and  $D'=\mathbb{B}^n\setminus\{a\}$.
The we have the following cases for $h_{D,c}\left(x,y\right)$:\\
If $\frac{1}{2}\leq|y|\leq|x|$, then
\begin{equation*}
h_{D,c}\left(x,y\right)=\log\left(1+\frac{c|x-y|}{\sqrt{1-|x|}\sqrt{1-|y|}}\right);
\end{equation*}
If $|y|\leq\frac{1}{2}\leq|x|$, then
\begin{equation*}
h_{D,c}\left(x,y\right)=\log\left(1+\frac{c|x-y|}{\sqrt{1-|x|}\sqrt{|y|}}\right);
\end{equation*}
If $|y|\leq|x|\leq\frac{1}{2}$, then
\begin{equation*}
h_{D,c}\left(x,y\right)=\log\left(1+\frac{c|x-y|}{\sqrt{|x|}\sqrt{|y|}}\right).
\end{equation*}
We also have
\begin{equation*}
h_{D',c}\left(\sigma_a\left(x\right),\sigma_a\left(y\right)\right)=\log\left(1+\frac{c|\sigma_a\left(x\right)-\sigma_a\left(y\right)|}{T}\right),
\end{equation*}
where
\begin{equation*}
T=\sqrt{\min\{|\sigma_a\left(x\right)-a|,1-|\sigma_a\left(x\right)|}\}\sqrt{\min\{|\sigma_a\left(y\right)-a|,1-|\sigma_a\left(y\right)|}\}.
\end{equation*}

Now we prove the inequality according to the following cases.

\textbf{Case 1}: $T=\sqrt{|\sigma_a\left(x\right)-a||\sigma_a\left(y\right)-a|}$.

Since
\begin{equation*}
|\sigma_a\left(x\right)-a|=|\sigma_a\left(x\right)-\sigma_a\left(0\right)|=\frac{s^2|x|}{|x-a^*||a^*|},\quad s^2=|a|^{-2}-1,
\end{equation*}
it follows that
\begin{equation*}
|\sigma_a\left(y\right)-a|=\frac{s^2|y|}{|y-a^*||a^*|}
\end{equation*}
and
\begin{equation*}
T=\frac{s^2}{|a^*|}\frac{\sqrt{|x||y|}}{\sqrt{|x-a^*||y-a^*|}}.
\end{equation*}
Hence
\begin{align*}
h_{D',c}\left(\sigma_a\left(x\right),\sigma_a\left(y\right)\right)
=\log\left(1+\frac{c|x-y|}{|a|\sqrt{|x||y||x-a^*||y-a^*|}}\right).
\end{align*}
By Lemma \ref{le 11}, we have
\begin{align*}
h_{D',c}\left(\sigma_a\left(x\right),\sigma_a\left(y\right)\right)
\leq\log\left(1+\frac{c|x-y|}{\sqrt{|x||y|\left(1-|a||x|\right)\left(1-|a||y|\right)}}\right).
\end{align*}

\textbf{Subcase 1.1}: $\frac{1}{2}\leq|y|\leq|x|$. By Lemma \ref{le 12} and Bernoulli's inequality \eqref{bql}, we have
\begin{align*}
h_{D',c}\left(\sigma_a\left(x\right),\sigma_a\left(y\right)\right)&\leq\log\left(1+\frac{c|x-y|}{\sqrt{|x||y|\left(1-|a||x|\right)\left(1-|a||y|\right)}}\right)\\
&\leq\log\left(1+\frac{c|x-y|}{\sqrt{\left(1-\frac{|a|}{2}\right)^2\left(1-|x|\right)\left(1-|y|\right)}}\right)\\
&\leq\frac{1}{1-\frac{|a|}{2}}\log\left(1+\frac{c|x-y|}{\sqrt{\left(1-|x|\right)\left(1-|y|\right)}}\right)\\
&=\frac{2}{2-|a|}h_{D,c}\left(x,y\right).
\end{align*}

\textbf{Subcase 1.2}: $|y|\leq\frac{1}{2}\leq|x|$. In this case, we have
\begin{align*}
h_{D',c}\left(\sigma_a\left(x\right),\sigma_a\left(y\right)\right)&\leq\log\left(1+\frac{c|x-y|}{\sqrt{|x||y|\left(1-|a||x|\right)\left(1-|a||y|\right)}}\right)\\
&\leq\log\left(1+\frac{c|x-y|}{\sqrt{\left(1-\frac{|a|}{2}\right)\left(1-|x|\right)|y|\left(1-\frac{|a|}{2}\right)}}\right)\\
&\leq\frac{1}{1-\frac{|a|}{2}}\log\left(1+\frac{c|x-y|}{\sqrt{\left(1-|x|\right)|y|}}\right)\\
&=\frac{2}{2-|a|}h_{D,c}\left(x,y\right).
\end{align*}

\textbf{Subcase 1.3}: $|y|\leq|x|\leq\frac{1}{2}$. We see that
\begin{align*}
h_{D',c}\left(\sigma_a\left(x\right),\sigma_a\left(y\right)\right)&\leq\log\left(1+\frac{c|x-y|}{\sqrt{|x||y|\left(1-|a||x|\right)\left(1-|a||y|\right)}}\right)\\
&\leq\log\left(1+\frac{c|x-y|}{\sqrt{|x||y|\left(1-\frac{|a|}{2}\right)\left(1-\frac{|a|}{2}\right)}}\right)\\
&\leq\frac{1}{1-\frac{|a|}{2}}\log\left(1+\frac{c|x-y|}{\sqrt{|x||y|}}\right)\\
&=\frac{2}{2-|a|}h_{D,c}\left(x,y\right).
\end{align*}

Hence, for Case 1, we obtain
$$h_{D',c}\left(\sigma_a\left(x\right),\sigma_a\left(y\right)\right)\leq\frac{2}{2-|a|}h_{D,c}\left(x,y\right).$$

\textbf{Case 2}:  $T=\sqrt{|\sigma_a\left(x\right)-a|\left(1-|\sigma_a\left(y\right)|\right)}$.

Since
\begin{align*}
|\sigma_a\left(y\right)|=|\sigma_a\left(y\right)-\sigma_a\left(a\right)|=\frac{s^2|y-a|}{|y-a^*||a-a^*|},
\end{align*}
we get
\begin{align*}
&h_{D',c}\left(\sigma_a\left(x\right),\sigma_a\left(y\right)\right)\\
&=\log\left(1+\frac{c|\sigma_a\left(x\right)-\sigma_a\left(y\right)|}{\sqrt{|\sigma_a\left(x\right)-a|\left(1-|\sigma_a\left(y\right)|\right)}}\right)\\
&=\log\left(1+\frac{cs^2|x-y|}{\sqrt{|x-a^*||y-a^*|}\sqrt{|x|\left(|y-a^*||a-a^*|-s^2|y-a|\right)}}\right)\\
&=\log\left(1+\frac{cs^2|x-y|}{\sqrt{|x-a^*||y-a^*|}\sqrt{|x||a-a^*|}\sqrt{|y-a^*|-|a^*||y-a|}}\right)\\
&=\log\left(1+\frac{c\left(1-|a|^2\right)|x-y|}{|a|\sqrt{|x-a^*||y-a^*|}\sqrt{|x|\left(1-|a|^2\right)}\sqrt{|a||y-a^*|-|y-a|}}\right)\\
&=\log\left(1+\frac{c\left(1-|a|^2\right)|x-y|\sqrt{|a||y-a^*|+|y-a|}}{|a|\sqrt{|x-a^*||y-a^*|}\sqrt{|x|\left(1-|a|^2\right)}\sqrt{\left(1-|a|^2\right)\left(1-|y|^2\right)}}\right)\\
&=\log\left(1+\frac{c|x-y|}{\sqrt{1-|y|^2}}\frac{\sqrt{|a||y-a^*|+|y-a|}}{\sqrt{|a||y-a^*|}}\frac{1}{\sqrt{|x||a||x-a^*|}}\right)\\
&\leq\log\left(1+\frac{c|x-y|}{\sqrt{1-|y|^2}}\frac{\sqrt{|a||y-a^*|+|y-a|}}{\sqrt{|a||y-a^*|}}\frac{1}{\sqrt{\left(1-|a||x|\right)|x|}}\right)\\
&=\log\left(1+\frac{c|x-y|}{\sqrt{1-|y|^2}}\sqrt{1+\frac{|y-a|}{|a||y-a^*|}}\frac{1}{\sqrt{\left(1-|a||x|\right)|x|}}\right)\\
&\leq\log\left(1+\frac{c|x-y|}{\sqrt{1-|y|^2}}\sqrt{1+\frac{|a|+|y|}{1+|a||y|}}\frac{1}{\sqrt{\left(1-|a||x|\right)|x|}}\right)\\
&=\log\left(1+\frac{c|x-y|}{\sqrt{1-|y|}}\sqrt{\frac{1+|a|}{1+|a||y|}}\frac{1}{\sqrt{\left(1-|a||x|\right)|x|}}\right).
\end{align*}

\textbf{Subcase 2.1}: $\frac{1}{2}\leq|y|\leq|x|$.  In this case, we have
\begin{align*}
&h_{D',c}\left(\sigma_a\left(x\right),\sigma_a\left(y\right)\right)\\
&\leq\log\left(1+\frac{c|x-y|}{\sqrt{1-|y|}}\sqrt{\frac{1+|a|}{1+|a||y|}}\frac{1}{\sqrt{\left(1-|a||x|\right)|x|}}\right)\\
&\leq\log\left(1+\frac{c|x-y|}{\sqrt{1-|y|}}\sqrt{\frac{1+|a|}{1+|a||y|}}\frac{1}{\sqrt{\left(1-\frac{|a|}{2}\right)\left(1-|x|\right)}}\right)\\
&\leq\log\left(1+\frac{c|x-y|}{\sqrt{\left(1-|y|\right)\left(1-|x|\right)}}\sqrt{\frac{2\left(1+|a|\right)}{2+|a|}}\sqrt{\frac{2}{2-|a|}}\right)\\
&=\log\left(1+\frac{c|x-y|}{\sqrt{\left(1-|x|\right)\left(1-|y|\right)}}2\sqrt{\frac{1+|a|}{4-|a|^2}}\right)\\
&\leq2\sqrt{\frac{1+|a|}{4-|a|^2}}h_{D,c}\left(x,y\right).
\end{align*}

\textbf{Subcase 2.2}: $|y|\leq\frac{1}{2}\leq|x|$. Similarly,
\begin{align*}
&h_{D',c}\left(\sigma_a\left(x\right),\sigma_a\left(y\right)\right)\\
&\leq\log\left(1+\frac{c|x-y|}{\sqrt{1-|y|}}\sqrt{\frac{1+|a|}{1+|a||y|}}\frac{1}{\sqrt{\left(1-|a||x|\right)|x|}}\right)\\
&\leq\log\left(1+\frac{c|x-y|}{\sqrt{|y|}}\sqrt{\frac{1+|a|}{1+|a||y|}}\frac{1}{\sqrt{\left(1-\frac{|a|}{2}\right)\left(1-|x|\right)}}\right)\\
&=\log\left(1+\frac{c|x-y|}{\sqrt{\left(1-|x|\right)|y|}}\sqrt{\frac{1+|a|}{1+|a||y|}}\sqrt{\frac{2}{2-|a|}}\right)\\
&\leq\log\left(1+\frac{c|x-y|}{\sqrt{\left(1-|x|\right)|y|}}\sqrt{1+|a|}\sqrt{\frac{2}{2-|a|}}\right)\\
&\leq\sqrt{\frac{2\left(1+|a|\right)}{2-|a|}}h_{D,c}\left(x,y\right).
\end{align*}

\textbf{Subcase 2.3}: $|y|\leq|x|\leq\frac{1}{2}$. We see that
\begin{align*}
&h_{D',c}\left(\sigma_a\left(x\right),\sigma_a\left(y\right)\right)\\
&\leq\log\left(1+\frac{c|x-y|}{\sqrt{1-|y|}}\sqrt{\frac{1+|a|}{1+|a||y|}}\frac{1}{\sqrt{\left(1-|a||x|\right)|x|}}\right)\\
&\leq\log\left(1+\frac{c|x-y|}{\sqrt{|y|}}\sqrt{\frac{1+|a|}{1+|a||y|}}\frac{1}{\sqrt{\left(1-\frac{|a|}{2}\right)|x|}}\right)\\
&=\log\left(1+\frac{c|x-y|}{\sqrt{|x||y|}}\sqrt{\frac{1+|a|}{1+|a||y|}}\sqrt{\frac{2}{2-|a|}}\right)\\
&\leq\log\left(1+\frac{c|x-y|}{\sqrt{|x||y|}}\sqrt{1+|a|}\sqrt{\frac{2}{2-|a|}}\right)\\
&\leq\sqrt{\frac{2\left(1+|a|\right)}{2-|a|}}h_{D,c}\left(x,y\right).
\end{align*}

Hence, for Case 2, we obtain
$$h_{D',c}\left(\sigma_a\left(x\right),\sigma_a\left(y\right)\right)\leq\sqrt{\frac{2\left(1+|a|\right)}{2-|a|}}h_{D,c}\left(x,y\right).$$

\textbf{Case 3}: $T=\sqrt{\left(1-|\sigma_a\left(x\right)|\right)|\sigma_a\left(y\right)-a|}$.
Since
\begin{align*}
|\sigma_a\left(x\right)|=|\sigma_a\left(x\right)-\sigma_a\left(a\right)|=\frac{s^2|x-a|}{|x-a^*||a-a^*|},
\end{align*}
Similarly, as in Case 2, we have
\begin{align*}
h_{D',c}\left(\sigma_a\left(x\right),\sigma_a\left(y\right)\right)
\leq\log\left(1+\frac{c|x-y|}{\sqrt{1-|x|}}\sqrt{\frac{1+|a|}{1+|a||x|}}\frac{1}{\sqrt{\left(1-|a||y|\right)|y|}}\right)
\end{align*}

\textbf{Subcase 3.1}: $\frac{1}{2}\leq|y|\leq|x|$. Then
\begin{align*}
&h_{D',c}\left(\sigma_a\left(x\right),\sigma_a\left(y\right)\right)\\
&\leq\log\left(1+\frac{c|x-y|}{\sqrt{1-|x|}}\sqrt{\frac{1+|a|}{1+|a||x|}}\frac{1}{\sqrt{\left(1-|a||y|\right)|y|}}\right)\\
&\leq\log\left(1+\frac{c|x-y|}{\sqrt{1-|x|}}\sqrt{\frac{1+|a|}{1+|a||x|}}\frac{1}{\sqrt{\left(1-\frac{|a|}{2}\right)\left(1-|y|\right)}}\right)\\
&=\log\left(1+\frac{c|x-y|}{\sqrt{\left(1-|y|\right)\left(1-|x|\right)}}\sqrt{\frac{1+|a|}{1+|a||x|}}\sqrt{\frac{2}{2-|a|}}\right)\\
&\leq\log\left(1+\frac{c|x-y|}{\sqrt{\left(1-|y|\right)\left(1-|x|\right)}}\sqrt{\frac{2\left(1+|a|\right)}{2+|a|}}\sqrt{\frac{2}{2-|a|}}\right)\\
&\leq2\sqrt{\frac{1+|a|}{4-|a|^2}}\log\left(1+\frac{c|x-y|}{\sqrt{\left(1-|x|\right)\left(1-|y|\right)}}\right)\\
&=2\sqrt{\frac{1+|a|}{4-|a|^2}}h_{D,c}\left(x,y\right).
\end{align*}

\textbf{Subcase 3.2}: $|y|\leq\frac{1}{2}\leq|x|$. Similarly,
\begin{align*}
&h_{D',c}\left(\sigma_a\left(x\right),\sigma_a\left(y\right)\right)\\
&\leq\log\left(1+\frac{c|x-y|}{\sqrt{1-|x|}}\sqrt{\frac{1+|a|}{1+|a||x|}}\frac{1}{\sqrt{\left(1-|a||y|\right)|y|}}\right)\\
&\leq\log\left(1+\frac{c|x-y|}{\sqrt{1-|x|}}\sqrt{\frac{1+|a|}{1+|a||x|}}\frac{1}{\sqrt{\left(1-\frac{|a|}{2}\right)|y|}}\right)\\
&=\log\left(1+\frac{c|x-y|}{\sqrt{|y|\left(1-|x|\right)}}\sqrt{\frac{1+|a|}{1+|a||x|}}\sqrt{\frac{2}{2-|a|}}\right)\\
&\leq\log\left(1+\frac{c|x-y|}{\sqrt{|y|\left(1-|x|\right)}}\sqrt{\frac{2\left(1+|a|\right)}{2+|a|}}\sqrt{\frac{2}{2-|a|}}\right)\\
&\leq2\sqrt{\frac{1+|a|}{4-|a|^2}}\log\left(1+\frac{c|x-y|}{\sqrt{|y|\left(1-|x|\right)}}\right)\\
&=2\sqrt{\frac{1+|a|}{4-|a|^2}}h_{D,c}\left(x,y\right).
\end{align*}

\textbf{Subcase 3.3}: $|y|\leq|x|\leq\frac{1}{2}$. We see that
\begin{align*}
&h_{D',c}\left(\sigma_a\left(x\right),\sigma_a\left(y\right)\right)\\
&\leq\log\left(1+\frac{c|x-y|}{\sqrt{1-|x|}}\sqrt{\frac{1+|a|}{1+|a||x|}}\frac{1}{\sqrt{\left(1-|a||y|\right)|y|}}\right)\\
&\leq\log\left(1+\frac{c|x-y|}{\sqrt{|x|}}\sqrt{\frac{1+|a|}{1+|a||x|}}\frac{1}{\sqrt{\left(1-\frac{|a|}{2}\right)|y|}}\right)\\
&\leq\log\left(1+\frac{c|x-y|}{\sqrt{|x||y|}}\sqrt{\frac{1+|a|}{1+|a||x|}}{\sqrt{\frac{2}{2-|a|}}}\right)\\
&\leq\log\left(1+\frac{c|x-y|}{\sqrt{|x||y|}}\sqrt{1+|a|}{\sqrt{\frac{2}{2-|a|}}}\right)\\
&\leq\sqrt{\frac{2\left(1+|a|\right)}{2-|a|}}\log\left(1+\frac{c|x-y|}{\sqrt{|x||y|}}\right)\\
&=\sqrt{\frac{2\left(1+|a|\right)}{2-|a|}}h_{D,c}\left(x,y\right).
\end{align*}

Therefore, for Case 3, we have
$$h_{D',c}\left(\sigma_a\left(x\right),\sigma_a\left(y\right)\right)\leq\sqrt{\frac{2\left(1+|a|\right)}{2-|a|}}h_{D,c}\left(x,y\right).$$

\textbf{Case 4}:  $T=\sqrt{1-|\sigma_a\left(x\right)|}\sqrt{|1-|\sigma_a\left(y\right)|}$. Then
\begin{align*}
&h_{D',c}\left(\sigma_a\left(x\right),\sigma_a\left(y\right)\right)\\
&=\log\left(1+\frac{c|\sigma_a\left(x\right)-\sigma_a\left(y\right)|}{\sqrt{1-|\sigma_a\left(x\right)|}\sqrt{|1-|\sigma_a\left(y\right)|}}\right)\\
&=\log\left(1+\frac{c\left(1-|a|^2\right)|x-y|}{|a|^2\sqrt{|x-a^*||y-a^*|}}\frac{1}{\sqrt{|x-a^*|-|a^*||x-a|}\sqrt{|y-a^*|-|a^*||y-a|}}\right)\\
&=\log\left(1+\frac{c\left(1-|a|^2\right)|x-y|}{|a|\sqrt{|x-a^*||y-a^*|}}\frac{1}{\sqrt{|a||x-a^*|-|x-a|}\sqrt{|a||y-a^*|-|y-a|}}\right)\\
&=\log\left(1+\frac{c\left(1-|a|^2\right)|x-y|}{|a|\sqrt{|x-a^*||y-a^*|}}\frac{\sqrt{|a||x-a^*|+|x-a|}\sqrt{|a||y-a^*|+|y-a|}}{\sqrt{\left(1-|a|^2\right)\left(1-|x|^2\right)}\sqrt{\left(1-|a|^2\right)\left(1-|y|^2\right)}}\right)\\
&=\log\left(1+\frac{c|x-y|}{\sqrt{\left(1-|x|^2\right)\left(1-|y|^2\right)}}\sqrt{1+\frac{|x-a|}{|a||x-a^*|}}\sqrt{1+\frac{|y-a|}{|a||y-a^*|}}\right)\\
&\leq\log\left(1+\frac{c|x-y|}{\sqrt{\left(1-|x|^2\right)\left(1-|y|^2\right)}}\sqrt{1+\frac{|x|+|a|}{1+|a||x|}}\sqrt{1+\frac{|y|+|a|}{1+|a||y|}}\right)\\
&=\log\left(1+\frac{c|x-y|}{\sqrt{\left(1-|x|\right)\left(1-|y|\right)}}\frac{1+|a|}{\sqrt{\left(1+|a||x|\right)\left(1+|a||y|\right)}}\right).
\end{align*}

\textbf{Subcase 4.1}: $\frac{1}{2}\leq|y|\leq|x|$. We see that
\begin{align*}
&h_{D',c}\left(\sigma_a\left(x\right),\sigma_a\left(y\right)\right)\\
&\leq\log\left(1+\frac{c|x-y|}{\sqrt{\left(1-|x|\right)\left(1-|y|\right)}}\frac{1+|a|}{\sqrt{\left(1+|a||x|\right)\left(1+|a||y|\right)}}\right)\\
&\leq\log\left(1+\frac{c|x-y|}{\sqrt{\left(1-|x|\right)\left(1-|y|\right)}}\frac{2\left(1+|a|\right)}{2+|a|}\right)\\
&\leq\frac{2\left(1+|a|\right)}{2+|a|}h_{D,c}\left(x,y\right).
\end{align*}

\textbf{Subcase 4.2}:  $|y|\leq\frac{1}{2}\leq|x|$. Similarly,
\begin{align*}
&h_{D',c}\left(\sigma_a\left(x\right),\sigma_a\left(y\right)\right)\\
&\leq\log\left(1+\frac{c|x-y|}{\sqrt{\left(1-|x|\right)\left(1-|y|\right)}}\frac{1+|a|}{\sqrt{\left(1+|a||x|\right)\left(1+|a||y|\right)}}\right)\\
&\leq\log\left(1+\frac{c|x-y|}{\sqrt{\left(1-|x|\right)|y|}}\frac{1+|a|}{\sqrt{1+\frac{|a|}{2}}}\right)\\
&\leq\frac{\sqrt{2}\left(1+|a|\right)}{\sqrt{2+|a|}}h_{D,c}\left(x,y\right).
\end{align*}

\textbf{Subcase 4.3}: $|y|\leq|x|\leq\frac{1}{2}$. We have
\begin{align*}
&h_{D',c}\left(\sigma_a\left(x\right),\sigma_a\left(y\right)\right)\\
&\leq\log\left(1+\frac{c|x-y|}{\sqrt{\left(1-|x|\right)\left(1-|y|\right)}}\frac{1+|a|}{\sqrt{\left(1+|a||x|\right)\left(1+|a||y|\right)}}\right)\\
&\leq\log\left(1+\frac{c|x-y|}{\sqrt{|x||y|}}\left(1+|a|\right)\right)\\
&\leq\left(1+|a|\right)h_{D,c}\left(x,y\right).
\end{align*}
Therefore, for Case 4, we have
$$h_{D',c}\left(\sigma_a\left(x\right),\sigma_a\left(y\right)\right)\leq\left(1+|a|\right)h_{D,c}\left(x,y\right).$$

Since $\max\left\{\frac{2}{2-|a|},\sqrt{\frac{2\left(1+|a|\right)}{2-|a|}},1+|a|\right\}=1+|a|$, we finally obtain
$$h_{D',c}\left(\sigma_a\left(x\right),\sigma_a\left(y\right)\right)\leq\left(1+|a|\right)h_{D,c}\left(x,y\right).$$

By Lemma \ref{le9}, $f^{-1}\left(x\right)=\left(A\circ\sigma_a\right)\left(x\right)$ and $f\left(x\right)=\left(\sigma_a\circ A^T\right)\left(x\right)$.
Then
\begin{align*}
h_{D',c}\left(f\left(x\right), f\left(y\right)\right)=&h_{D',c}\left(\left(\sigma_a\circ A^T\right)\left(x\right),\left(\sigma_a\circ A^T\right)\left(y\right)\right)\\
\leq&\left(1+|a|\right)h_{D,c}\left(A^T\left(x\right),A^T\left(y\right)\right)\\
=&\left(1+|a|\right)h_{D,c}\left(x,y\right).
\end{align*}
This completes the proof.\hfill$\Box$


{\small
{\em Authors' addresses}:
{\em Yinping Wu}, {\em Gendi Wang(Corresponding author)}, {\em Gaili Jia}, {\em Xiaohui Zhang},
  Department of Mathematical Sciences,
  Zhejiang Sci-Tech University, 310018 Hangzhou, China.
 e-mail: \texttt{yinping\_wu95@\allowbreak 163.com, gendi.wang@\allowbreak zstu.edu.cn, gaili\_jia@\allowbreak 163.com, xiaohui.zhang@\allowbreak zstu.edu.cn}.
}

\end{document}